\documentclass[12pt]{article}

\usepackage{amsfonts}
\usepackage{graphics}
\usepackage{amssymb}

\newtheorem{theorem}{Theorem}[section]

\newtheorem{definition}{Definition}[section]

\newtheorem{remark}{Remark}[section]

\begin{document}

\begin{center}

{\Large Strong solution of the stochastic Burgers equation\\}

\end{center}

\vspace{0.3cm}

\begin{center}
{\large P. Catuogno  and
C. Olivera \footnote{Supported by FAEPEX under grant no. 1324/12}} \\

\textit{Departamento de
 Matem\'{a}tica, Universidade Estadual de Campinas,  \\ F. 54(19) 3521-5921 - Fax 54(19)
3521-6094\\ 13.081-970 -
 Campinas - SP, Brazil. e-mail: pedrojc@ime.unicamp.br ; colivera@ime.unicamp.br}
\end{center}
\vspace{0.3cm}

\begin{center}
\begin{abstract}

\noindent This work introduces a  pathwise
notion of solution    for the stochastic Burgers equation, in particular, our approach encompasses the Cole-Hopf solution.  
The developments are based on  regularization arguments from the theory of distributions.
\end{abstract}
\end{center}

\noindent {\bf Key words:} Stochastic Burgers equation, Generalized functions, Generalized stochastic processes, Colombeau algebras, 
Stochastic partial differential equations.

\vspace{0.3cm} \noindent {\bf MSC2000 subject classification:} 60H15, 46F99.

\section {Introduction}

\noindent

The aim of this paper is to study the existence  of solution to the 
stochastic Burgers equation of the following form:

\begin{equation}\label{burg}
 \left \{
\begin{array}{lll}
    \partial_t U(t, x) =  \ \triangle U(t, x) +    \ \partial_x   U^{2}(t,x) +   \partial_x  W(t,x),\\
 U(0, x) = \partial_x  f(x).
\end{array}
\right .
\end{equation}

\noindent where $W(t,x)$ is a space-time white noise.  

\noindent  There are many motivations for studying the stochastic Burgers equation. Several authors have indeed suggested to use the stochastic Burgers 
equation as a simple model for turbulence,  has also been proposed to study the dynamics of interfaces and 
numerous applications were found in astrophysics and statistical physics, see for example 
\cite{Beck} and  \cite{Wie}. We refer to \cite{Bre} and \cite{Hairer2} for a more detailed historical account of the stochastic Burgers equation
 from a purely mathematical point of view.


\noindent The main difficulty with the stochastic Burgers equation (\ref{burg}) is that the solutions do not take values in a function space but in 
a generalized function space. 
Thus it is necessary to give meaning to the non-linear term $\partial_x U^2$, because the usual
product makes no sense for arbitrary distributions. We recall that is not possible to define a product for arbitrary 
distributions with good properties, see \cite{colb} and \cite{ober2}.

\noindent In this article, we deal with product of
distributions via regularizations. This is, we approximate the distributions to be multiplied by smooth functions,
multiply the approximations and pass to
the limit (see for instance \cite{Miku},  \cite{calo} and  \cite{ober2}.

\noindent L. Bertini and G. Giacomin in \cite{BG}, proposed that $U(t, x)=\partial_x \ln Z(t,x)$ is the meaningful solution to the stochastic Burgers 
equation. Here $Z(t,x)$ denotes the solution of the stochastic heat equation with multiplicative noise. This is known as the Cole-Hopf solution for 
the stochastic Burgers equation. We recall that the stochastic heat equation with multiplicative noise is the It\^o equation,
\begin{equation}\label{heat}
 \left \{
\begin{array}{lll}
 dZ   =  \triangle Z \ dt  + Z dW \\
 Z(0, x) = e^{f(x)}.
\end{array}
\right .
 \end{equation}

\noindent The Hopf-Cole solution is believed to be the correct physical solution for (\ref{burg}). However
up to recently a rigorous notion of solution to the stochastic Burgers equation was lacking.  In \cite{AS2}, S. Assing introduces a weak solution in a probabilistic sense for the   equation (\ref{burg}). The idea is to approximate the Cole-Hopf solution by the density fluctuations in weakly asymmetric exclusion. In \cite{Jara}, P. Gon\c{c}alves and M. Jara considered a similar type of solution.

\noindent In this article we introduce a nice space of generalized stochastic processes in order to prove that
the Cole-Hopf solution is a strong solution to the  equation (\ref{burg}). We also show that the Cole-Hopf solution satisfies a certain type of stability.
   
Our  space of generalized stochastic processes 
looks like a generalized stochastic process space in the sense  of It\^o-Gelfand-Vilenkin, see  \cite{Gel4}  and \cite{ITO}. A main ingredient in our 
approach is the  use of regularization techniques to define nonlinear operations
in the theory of distributions, we refer the reader to  \cite{Miku} and \cite{ober2} for the background material.

\noindent {\bf 
Finally, we mention  that the Cole-Hopf solution
does not make sense in the  multidimensional case  since the solution of the  stochastic heat equation
is not a standard  stochastic process. It is realized as a generalized stochastic process in the space of stochastic Hida distribution, see for example \cite{HOBZ}. Thus, we  expect solutions of  Stochastic Burgers equation
  only in the sense of  Colombeau's generalized functions. We refer the reader to
  \cite{albe}, \cite{CO1}, \cite{CO2}, \cite{ober4}  and \cite{Russo} for applications of this theory in stochastic analysis. 
 }

\noindent The article is organized as follows: Section 2 reviews some basic facts on the standard cylindrical Wiener process. Section 3, we give a 
notion of random generalized functions, we introduce a new concept of solution for the stochastic Burgers equation and we prove that the Cole-Hopf solution solves  (\ref{burg}). Also we prove that its has certain property  of stability.


\section{Cylindrical Wiener process }

\noindent Let   $\{W(t,\cdot): t\in [0,T]\}$ be a
standard cylindrical Wiener process in $L^{2}(\mathbb{R})$; it is canonically realized as a family of
continuous processes satisfying:
\begin{enumerate}
\item For any $\varphi \in L^2( \mathbb{R})$, $\{W_t (\varphi), t \in [0,T] \}$ is a Brownian motion with variance $t\int\varphi^2(x)~dx$,
\item For any $\varphi_1,\varphi_2 \in L^2( \mathbb{R})$ and  $s, t \in [0,T]$,
\[
\mathbb{E}(W_{s}(\varphi_1)(W_{t}(\varphi_2))=(s \wedge t) \int \varphi_1(x) \varphi_2(x) ~dx.
\]
\end{enumerate}

\noindent Let $\{ \mathcal{F}_t : t \in [0,T]\}$ be the $\sigma$-field generated by the $P$-null sets and the random variables
$W_{s}(\varphi)$, where $\varphi\in \mathcal{D}( \mathbb{R})$ and  $s \in [0,t]$. The predictable $\sigma$-field is the $\sigma$-field in
$[0, T] \times \Omega$  generated by the sets $(s, t]\times A$ where $A\in\mathcal{F}_s$ and $0 \leq s < t \leq T$.

\noindent Let $\{v_j : j\in \mathbb{N} \}$ be a complete orthonormal basis of  $L^{2}(\mathbb{R})$. For any predictable
process $g \in L^{2}(\Omega \times [0, T ], L^2( \mathbb{R}) )$ it turns out that the following series is convergent in
$L^{2}(\Omega, \mathcal{F}, P)$ and the sum does not depend on the chosen orthonormal system:

\begin{equation}\label{inteCy}
\int_{0}^{T}  g_{t} W_{t}  := \sum_{j=1}^{\infty} \int_{0}^{T}  (g_t, v_{j}) dW_t(v_{j}).
\end{equation}

\noindent  We notice that each summand in the above series is a classical It\^o integral with respect to a standard
Brownian motion, and the resulting stochastic integral is a real-valued random variable. The independence of the terms in the series
(\ref{inteCy}) leads to the isometry property

\[
\mathbb{E}(| \int_{0}^{T}  g_{s} ~ dW_{s} |^{2})
   = \mathbb{E}( \int_{0}^{T}  \int | g_s(x) |^{2}~dx~ds).
\]

\noindent See   \cite{Da} for properties of the cylindrical Wiener process and
stochastic integration.

\noindent The mollifier cylindrical Wiener process is defined by:
\begin{equation}
W_{t}^{n}(x):= W_{t}(n\rho(n(x-\cdot)))
\end{equation}
where $\rho :  \mathbb{R} \rightarrow \mathbb{R}$ is an infinitely differentiable function with
compact support such that $\int \rho(x) \ dx=1$. 

The mollifier Wiener process satisfies:
\begin{enumerate} 
\item The covariance of the mollifier Wiener process is given by 
\begin{equation}\label{reguCy}
\mathbb{E}[W_{t}^{n}(x) W_{s}^{n}(y)]= s\wedge t  \int C_n(x-y)
\end{equation}
where $C_n(z)=\int \delta_{n}(z-u) \delta_{n}(-u) du$ and $\delta_n(z)=n \rho (nz)$.

\item The quadratic variation of $W_{t}^{n}(x)$ is given by
\begin{equation}\label{reguCy1}
< W^{n}(x)>_t= Cnt
\end{equation}
where $C =  \int \rho^2(-x)dx$.

\item The mollifier Wiener process is an approximation to the cylindrical Wiener process. For all $\varphi \in L^2(\mathbb{R})$,
\begin{equation}
 \lim_{n \rightarrow \infty} \int W^n_t(x) \varphi(x)dx = W_t(\varphi).
\end{equation}
In the case that $\varphi$ has compact support the above convergence is a.e..

\end{enumerate}




\section{Solving the stochastic Burgers equation}


\noindent We denote by $\mathcal{D}((0,T) \times \mathbb{R})$ the space of
the infinitely differentiable functions with
compact support in   $(0,T) \times \mathbb{R}$,  and
$\mathcal{D}^{\prime}((0,T) \times  \mathbb{R})$ its dual.

\begin{definition} Let $\mathbf{D}$ be the  space of functions $T :  \Omega \ \rightarrow \  \mathcal{D}^{\prime}((0,T) \times \mathbb{R})$ such that
$<T,\varphi> $ is  a  random   variable for all  $\varphi \in   \mathcal{D}((0,T) \times \mathbb{R})$.    \
The elements of $\mathbf{D}$ are called random generalized
functions.
\end{definition}

The initial condition of the stochastic Burgers equation requires the usage of the notion of  section of a distribution in the sense of S. Lojasiewicz, see \cite{lo} and \cite{ober2}. For convenience of the reader, we present the relevant definitions.


\begin{definition}  A strict delta net is a net $\{  \rho_{\varepsilon} : \varepsilon >0 \}$ of $\mathcal{D}((0,T))$ such that it satisfies:
 \begin{enumerate}
\item
$\lim_{\varepsilon \rightarrow 0} supp(\rho_{\epsilon})= \{ 0 \}$.

\item For all $\varepsilon >0$,
 $\int \rho_{\varepsilon}(t)  \ dt =1$.
\item $\sup_{\varepsilon >0} \int |\rho^{\varepsilon}(t)| \ dt < \infty$.
\end{enumerate}
\end{definition}

 
\begin{definition} A distribution $H   \in  \mathcal{D}^{\prime}((0,T) \times \mathbb{R})$ has a section
$U \in \mathcal{D}^{\prime}(  \mathbb{R}) $ at $t=0$ if for all $\varphi \in \mathcal{D}( \mathbb{R})$ and
all strict delta net $\{  \rho_{\varepsilon} : \varepsilon >0 \}$,

\[
\lim_{\varepsilon \rightarrow 0}<H,  \rho_{\epsilon} \varphi > \ = \ <U,  \varphi >.
\]
\end{definition}

\subsection{Existence of generalized solution}
 {\bf 
\noindent We say that a random field $\{ S(t,x) : t \in [0,T],~x \in \mathbb{R} \}$ is a spatially dependent semimartingale if for
each $x\in \mathbb{R}$,
$\{S(t,x): t \in [0,T]\}$ is a semimartingale in relation to the same filtration $\{ \mathcal{F}_t : t \in [0,T] \}$. If
$S(t,x)$ is a $C^{\infty}$-function of $x$ and continuous in $t$ almost everywhere,  it is called  a $C^{\infty}$-semimartingale.  See \cite{Ku} for a rigorous study of spatially depend semimartingales
and applications to  stochastic differential equations.
} Now, following ideas of regularization and passage to the limit, we
introduce a new concept of solution for the stochastic Burgers equation.

\begin{definition}\label{solu3}
We say that $U\in  \mathbf{D}$
is a {\it generalized  solution} of the equation  (\ref{burg})
if
\begin{enumerate}
\item There exists a sequence of $C^{\infty}$-semimartingales $\{ U_n : n \in \mathbb{N} \}$ such that
$ U=\lim_{n \rightarrow \infty}U_n$
and there exists
$\lim_{n\rightarrow\infty}   \partial_x U_n^{2}$
in $\mathcal{D}^{\prime}((0,T) \times  \mathbb{R})$   {\bf  almost surely for $\omega \in \Omega$}.

\item For all $\varphi \in \mathcal{D}((0,T) \times  \mathbb{R})$,
\[
< U, \partial_t\varphi > +   < \triangle U, \varphi >  +  < \partial_x  U ^{2},  \varphi >  +  \int_{0}^{T}   \partial_x \varphi(t,\cdot) dW_t=0
\]
where $\partial_x U ^{2}:=\lim_{n\rightarrow\infty} \partial_x U_n^{2}$.

\item {\bf There exists a section of $U$ at
$t = 0$ and is equal  to $\partial_x f$ }.
\end{enumerate}

\end{definition}

\begin{theorem}\label{teoSolu2} Let  $f\in  C_{b}^{\infty}(\mathbb{R})$ and $Z$ be a solution of the stochastic heat equation (\ref{heat}). 
Then $U=\partial_x \ln Z$ is a generalized solution of the stochastic Burgers equation (\ref{burg}).

\end{theorem}

\begin{proof}  Let us denote by $H_{n}(t,x)$ the process $\ln   Z_{n}(t,x)$, where $Z_n$ is the solution of the regularized stochastic heat equation in
the It\^o sense
\begin{equation}\label{heat2}
 \left \{
\begin{array}{lll}
dZ_{n}  & = & \triangle Z_{n} \ dt  + Z_{n} \ dW^{n}, \  \\
Z_n(0,x) & = &  e^{f(x)}.
\end{array}
\right .
\end{equation}
 {\bf 
We observe that the solution of the equation (\ref{heat2}) is understood in a mild sense, this is, $Z_{n}$ satisfies the equation

\[
 Z_{n} =  G_t\ast f 
+   \int_{0}^{t}  (G_{t-s}\ast Z_{n})  dW_{t}^{n} 
\]

\noindent where   $G_t=G(t,.)$ is   the fundamental solution  the heat equation. This construction is due to L. Bertini and G. Giacomin \cite{BC} pp. 1884. 
 The solution $Z(t,x)$ of the stochastic heat equation (\ref{heat}) is too  understood in a mild sense, see  Definition 2.1  and Theorem 2.2 of \cite{BC}.
}

\noindent By It\^o formula and (\ref{reguCy1}) we have

\begin{equation}\label{KPZA2}
H_{n} =   f  +
 \int_{0}^{t}  \triangle H_{n} \ ds
+   \int_{0}^{t} (\partial_x H_{n})^{2} \ ds
+     W_{t}^{n} - \frac{ C}{2} \ nt .
\end{equation}

\noindent Let $U_{n}(t,x) = \partial_x \ln  Z_{n}(t,x)$. Multiplying (\ref{KPZA2}) by $\partial_x  \partial_t \varphi(t,x)$, where $\varphi \in \mathcal{D}((0,T) \times  \mathbb{R})$,
and integrating in $(0,T)\times \mathbb{R}$ we obtain that
\begin{equation}\label{KPZA3}
< U_{n}, \partial_t \varphi > +
   < \triangle U_{n}, \varphi > +  < \partial_x  U_{n}^{2},  \varphi >
+    \int_{0}^{T} ( \partial_x \varphi(t,\cdot) * \delta_{n} )  \  dW_{t} =0 .
\end{equation}
We observe that   $ Z_{n}$ converge to  $Z$    uniformly on compacts of $(0,T) \times \mathbb{R}$, see  
L. Bertini and N. Cancrini \cite{BC}, Theorem 2.2. Thus

  \begin{equation}\label{C1}
 \lim_{n \rightarrow \infty}< U_{n}, \partial_t \varphi> = < U, \partial_t \varphi>
   \end{equation}
and
  \begin{equation}\label{C2}
\lim_{n \rightarrow \infty}< \triangle U_{n}, \varphi> =   < \triangle U, \varphi >.
\end{equation}

\noindent We recall that $\int_{0}^{T}   \varphi(t,\cdot) \ dW_{t}$ defines a continuous linear functional from 
$\mathcal{D}((0,T) \times \mathbb{R})$ to $\mathbb{R}$, see  K. Schaumloffel \cite{Sch}. Then

\begin{equation}\label{C3}
\lim_{n \rightarrow \infty}  \int_{0}^{T} (\partial_x\varphi(t,\cdot) * \delta_{n})  \  dW_{t} = \int_{0}^{T}   \partial_x \varphi \ dW_{t}.
\end{equation}

\noindent  From the equation (\ref{KPZA3})
and the  convergences   (\ref{C1}), (\ref{C2}) and (\ref{C3}) we have that for all $ \varphi\in \mathcal{D}[0,T) \times \mathbb{R})$,

\[
 \int_{0}^{T} \int_{ \mathbb{R}}   \partial_x U_{n}^{2} \  \varphi(t,x)  \ dtdx \
\]
converges and defines a linear functional in $\mathcal{D}^{\prime}((0,T) \times \mathbb{R})$ . Thus the  nonlinearity
\[
<\partial_x ( U )^{2} , \varphi> := \lim_{n \rightarrow \infty} \int_{0}^{T} \int_{ \mathbb{R}}   \partial_x ( U_{n})^{2} \  \varphi(t,x)  \ dtdx \
\]
is well defined.

\noindent From the  continuity of $ Z(t,x) $, see for instance \cite{BC},  we observe that

\begin{eqnarray*}
\lim_{\varepsilon\rightarrow 0} \int \int_{0}^{T} \partial_x \ln \  Z(t,x) \  \rho_{\varepsilon}(t) \ dt \   \varphi(x) \ dx  & = & 
-  \int f(x) \  \partial_x \varphi(x)\ dx \\
& = & \int   \partial_x  f(x) \   \varphi(x)\ dx    
\end{eqnarray*}
for all $\varphi \in \mathcal{D}( \mathbb{R})$ and for all strict delta net $\{ \rho_{\varepsilon}: \varepsilon >0 \}$. Thus  we conclude that $U$ is a generalized  solution   for the  problem (\ref{burg}).

\end{proof}

 {\bf 
\begin{remark}
 In our formulation of the equation  (\ref{burg}) the initial condition is $U |_{\{t=0\}} = \partial_x  f(x)$. Thus if $f$ and $g$ are functions such that 
$ \partial_x  f(x)= \partial_x  g(x)$  the solution of equation (\ref{burg}) are the same. This is, we are identifying
inicial conditions that have the same derivative. 
\end{remark}
}

 {\bf 
\begin{remark} The uniqueness problem is interesting and probably impossible to solve in this setting.
The problem is that different approximations of the stochatic heat equations could have different limits, we refer to the reader to section 2.4
of \cite{Hairer}.
\end{remark}
}

\subsection{Stability of the Cole-Hopf solution}

A well-known  statement is that if a PDE is well-posed, then the
solutions to every reasonable sequence of approximating PDEs  should
converge to it. However, in SPDEs this  statement is ambiguous  since
stochastic integrals are too irregular to be defined pathwise in an unambiguous way. See \cite{Hairer3} for interesting comments in relation to 
this issue in SPDEs.

\noindent The Cole-Hopf solution of the stochastic Burgers equation has this property of stability for approximations obtained as solutions of the  
stochastic Burgers equation driven by a  mollifier Wiener process. 













\begin{theorem} Let  $f\in  C_{b}^{\infty}(\mathbb{R})$. Then the Cole-Hopf solution $U=\partial_x \ln~Z$ of the stochastic Burgers equation with initial condition 
$\partial_x f$ satisfies the following stability property:

\noindent If $V_n$ are semimartingales such that
\[
 V_{n} =   \partial_x f  + \int_{0}^{t}  \triangle V_{n} \ ds
 \  \int_{0}^{t} \partial_x  V_{n}^{2} \ ds
+  \partial_x   W_{t}^{n} 
\]
then $\lim_{n \rightarrow \infty} V_{n}=U$ in $\mathcal{D}^{\prime}((0,T) \times  \mathbb{R})$.

\end{theorem}

\begin{proof}  We observe that $ V_{n}=  \partial_x  H_n$ where 
$ H_n$ satisfies

\[
 H_{n} =    f  + \int_{0}^{t}  \triangle H_{n} \ ds
 \  \int_{0}^{t} (  \partial_x  H_{n})^{2} \ ds
+     W_{t}^{n} .
\]

We have that $G_n=e^{H_{n}}$  verifies the 
 following  Stratonovich equation 

\begin{equation}\label{heatStra}
 \left \{
\begin{array}{lll}
dG_{n}  & = & \triangle G_{n} \ dt  + G_{n} \circ dW^{n}   \\
G_n(0,x) & = &  e^{f(x)},
\end{array}
\right .
\end{equation}

\noindent or equivalently  the It\^o equation

\begin{equation}\label{heatito3}
 \left \{
\begin{array}{lll}
dG_{n}  & = & \triangle G_{n} \ dt  + G_{n}  dW^{n} + n \ C \  G_{n}  dt \  \\
G_n(0,x) & = &  e^{f(x)}.
\end{array}
\right .
\end{equation}

\noindent A trivial calculation shows that $ G_n =  Z_n e^{C n t}$ where $ Z_n $  is the solution of the equation (\ref{heat2}). Then

\[
H_{n} =\ln \  \left(Z_n e^{C\ nt}\right).
\]

Thus

\[
V_n =  \partial_x H_n = \partial_x \ln \ Z_n. 
\]

\noindent By continuity,  we have that $V_n$ converges to  $U=\partial_x \ln~Z$ in  $\mathcal{D}^{\prime}((0,T) \times  \mathbb{R})$.

\end{proof}





\begin{thebibliography}{9999}
\bibitem{albe} S. Albeverio, Z. Haba, F. Russo. {\it A two-space
dimensional semilinear heat equation perturbed by (Gaussian) white
noise}. Probab. Theory Related Fields. 2001, 121, 319-366.


\bibitem{Miku}
P. Antosik, J. Mikusi\'nski, R. Sikorski. {\it Theory of
distributions. The sequential approach}. Elsevier Scientific
Publishing Company, 1973.



\bibitem{AS}
S. Assing.  {\it A pregenerator for Burgers equation forced by conservative noise}.  Commun.
Math. Phys. 2002, 225, 611-632.

\bibitem{AS2}
S. Assing. {\it A rigorous equation for the Cole-Hopf solution of the
conservative KPZ dynamics}. 2012, arXiv:1109.2886 .



\bibitem{Beck}
J. Bec,  K. Khanin. {\it Burgers turbulence}.  Phys. Rep. 2007, 447, 1-66.




\bibitem{BG}
L. Bertini, G. Giacomin. {\it Stochastic Burgers and KPZ equations from particle
systems}. Comm. Math. Phys. 1997, 183, 571-607.






\bibitem{BC}
L. Bertini, N. Cancrini.  {\it The stochastic heat equation: Feynman-Kac formula and intermittence}.
Journal of Statistical Physics. 1995, 78, 5-6, 1377-1401.



\bibitem{Bre}
Z. Brzezniak, B. Goldys, M. Neklyudov. {\it Multidimensional stochastic Burgers equation}, 2012, arXiv:1202.3230v1. 








\bibitem{CO1}
P. Catuogno,  C. Olivera. {\it Tempered generalized functions and
Hermite expansions}. Nonlinear Analysis. 2012, 74, 479-493.


\bibitem{CO2}
P. Catuogno,  C. Olivera.{\it On Stochastic generalized
functions}. Infinite Dimensional Analysis, Quantum
Probability and Related Topics. 2011,  14, 2, 237-260.



\bibitem{calo}
P. Catuogno, S. Molina,  C. Olivera. {\it On Hermite
representation of distributions and products}. Integral Transforms
and Special Functions. 2007 18, 4, 233-243.




\bibitem{colb}
J. Colombeau. {\it Elementary introduction to new generalized
functions}. Math. Studies 113, North Holland, 1985




\bibitem{Da}
R. Dalang, L. Quer-Sardanyons. {\it Stochastic integrals for spde's: a comparison}, Expositiones Mathematicae. 2011,
29, 1, 67-109.


\bibitem{Gel4}
I. Gelfand, N. Vilenkin. {\it Generalized functions IV }. Academic Press, 1964.


\bibitem{Jara}
P. Goncalvez, M. Jara. {\it Universality of KPZ equation}. 2010, arXiv:1003.4478v1.




\bibitem{Hairer3}
M. Hairer.{\it  Singular perturbations to semilinear stochastic heat equations }. Probability Theory and Related Fields. 2012,
 152, 1-2, 265-297.


\bibitem{Hairer2}
M. Hairer, J. Voss. {\it  Approximations to the stochastic Burgers equation}. J. Nonlinear Sci. 2011, 21, 6, 897-920.

\bibitem{Hairer}
M. Hairer: {\it Solving the KPZ equation}. arXiv:1109.6811v3, Ann. Math. (2013). To appear..

\bibitem{HOBZ}
H. Holden, B. Oksendal, J. Uboe,  T. Zhang. {\it Stochastic
partial differential equations. A modeling, white noise functional
approach}. Birkhauser, 1996.



\bibitem{ITO}
K. It\^o. {\it Stationary random distributions}. Mem. Coll. Sci. Kyoto
Univ. Ser. A. 1954, 28, 209-223.







\bibitem{Ku}
H. Kunita: {\it Stochastic flows and stochastic differential
equations}. Cambridge University Press, 1990.







\bibitem{lo}
S. Lojasiewicz. {\it Sur la valeur et la limite d\'une distribution en un point}.
Studia Math. 1957, 16, 1-36.




\bibitem{ober2}
M. Oberguggenberger. {\it Multiplication of distributions and
applications to partial differential equations}.  Pitman Research
Notes in Math. Series 259. Ed. Longman Science and Technology, 1993.















\bibitem{ober4} M. Oberguggenberger,  F. Russo. {\it Nonlinear SPDEs:
Colombeau solutions and pathwise limits.} Stochastic analysis and
related topics, VI (Geilo, 1996),  Progr. Probab., 42,
Birkhauser Boston, Boston, MA,  1998.

\bibitem{Russo}
F. Russo  {\it Colombeau generalized functions and stochastic
analysis}. Stochastic analysis and applications in physics. Edit. A. Cardoso, M. de  Faria, J. Potthoff, R. Seneor,
L. Streit. Adv. Sci. Inst. Ser. C Math. Phys. Sci., Kluwer Acad. Publ. 1994.


\bibitem{Sch}
K. Schaumloffel .{\it White noise in space and time as the time-derivative of a cylindrical Wiener process }.
Stochastic Partial Differential Equations and Applications II
Lecture Notes in Mathematics. 1989, 1390, 225-229.


\bibitem{Wie}
E. Weinan. {\it  Stochastic hydrodynamics}, Current developments in mathematics,  109-147, Int. Press, Somerville, 2000.



\end{thebibliography}
\end{document}